\documentclass[12pt]{amsart}
\usepackage{amsmath, amscd, amssymb}
\usepackage{graphpap, color}
\usepackage[mathscr]{eucal}
\usepackage{mathrsfs}
\usepackage{pstricks}
\usepackage{color}
\usepackage{cancel}
\usepackage[mathscr]{eucal}
\usepackage{stmaryrd}

\usepackage[margin=2cm]{geometry}
\usepackage {hyperref}

\usepackage{pstricks}
\usepackage{color}
\usepackage{cancel}
\usepackage{verbatim}

\numberwithin{equation}{section}

\def\hb{\hbar}

\def\sO{{\mathscr O}}

\def\sO{\mathscr{O}}

\newcommand{\CC}{\mathbb{C}}
\newcommand{\EE}{\mathbb{E}}

\newcommand{\PP}{\mathbb{P}}
\newcommand{\QQ}{\mathbb{Q}}

\newcommand{\ZZ}{\mathbb{Z}}

\newcommand{\FF}{\mathbb{F}}

\newcommand{\TT}{\mathbb{T}}

\def\fR{{\mathfrak R}}

\def\cA{{\mathcal A}}

\def\AA{\mathbb A}

\newcommand{\cal}{\mathcal}

\def\cA{{\cal A}}
\def\cB{{\cal B}}
\def\cC{{\cal C}}

\def\cF{{\cal F}}

\def\cM{{\cal M}}
\def\cN{{\cal N}}

\def\cP{{\cal P}}

\def\cV{{\cal V}}

\def\cS{{\cal S}}

\def\cY{{\cal Y}}
\def\cZ{{\cal Z}}

\def\fC{\mathfrak{C}}
\def\fD{\mathfrak{D}}

\def\lra{\longrightarrow}

\newcommand{\si}{\sigma}

\def\begeq{\begin{equation}}
\def\endeq{\end{equation}}
\def\and{\quad{\rm and}\quad}

\def\and{\quad\text{and}\quad}

 \DeclareMathOperator{\val}{val}

\DeclareMathOperator{\Ver}{{Ver}}
\DeclareMathOperator{\Edg}{{Edg}}
\newtheorem{prop}{Proposition}[section]
\newtheorem{theo}[prop]{Theorem}
\newtheorem{lemm}[prop]{Lemma}
\newtheorem{coro}[prop]{Corollary}

\newtheorem{defi-prop}[prop]{Definition-Proposition}

\def\ev{\text{ev}}

\def\sO{{\mathscr O}}

\def\beq{\begin{equation}}
\def\eeq{\end{equation}}

\def\bM{\mathbf{M}}

\def\bD{{\mathbf D}}

\def\bee{\begin{equation}}
\def\eeq{\end{equation}}

\def\a{\mathbf a}
\def\e{\mbox{e}}

\date{}

\date{}

\title[Genus one stable quasimap invariants]{Genus one stable quasimap invariants for projective complete intersections}

\author[M.-L.~Li]{Mu-lin Li}
\address{College of Mathematics and Econometrics, Hunan University, China} \email{mulin@hnu.edu.cn}

\begin{document}
\begin{abstract}
By using the infinitesimally marking point to break the loop in the localization calculation as Kim and Lho, and Zinger's explicit formulas for double $J$-functions, we obtain a formula for genus one stable quasimaps invariants when the target is a complete intersection Calabi-Yau in projective space, which gives a new proof of Kim and Lho's mirror theorem for elliptic quasimap invariants.
\end{abstract}

\maketitle

\section{Introduction}
The moduli space of stable quotients was first constructed and studied by Marian, Oprea and Pandharipande \cite{MOP}. Cooper and Zinger \cite{CZ} calculated the $J$-function of stable quotients for projective complete intersections, and proved that it is related to the genus zero Gromov-Witten invariants by mirror map. Using the moduli space of stable quasimaps, which is a generalization of stable quotient by Ciocan-Fontanine, Kim and Maulik~\cite{FKM}, Ciocan-Fontanine, Kim proved that the genus zero stable quasimap invariants (including twisted cases) are related to stable map by mirror map in \cite{FK}. Later Kim and Lho \cite{KB} obtained the formula for genus one stable quasimap invariants without markings for projective complete intersections. Combining with the result in \cite{FK2} it recovered the results of Popa \cite{Popa} and Zinger \cite{Zin1} on genus one Gromov Witten invariants of projective complete intersections by mirror map. In this paper using the double $J$-functions of complete intersections in projective space proved by Zinger \cite{Zin2}, we give another proof of that formula.

Let $\ell$ be a nonnegative integer. $E:=\oplus_{i=1}^{\ell}L_i$ be the direct sum of line bundles over $\PP^{n-1}$ with degree $a_i=\deg L_i>0$.
%\begin{gather*}
%|\a|=\sum_{k=1}^{\ell}a_k,  \qquad
%\langle\a\rangle=\prod_{k}\!a_k\,, \qquad
%\a!=\prod_{k}\!\!a_k!\,, \qquad
%\a^{\a}=\prod_{k=1}^{\ell} a_k^{a_k}.
%\end{gather*}

If $|\a|:=\sum_{k=1}^{\ell}a_k=n$, a transversal section $s$ of $E$ gives a Calabi-Yau manifold $X$, which is a complete intersection in $\PP^{n-1}$.

Let $\overline{Q}_{1,0}(X,d)$ be the moduli space of genus one and degree $d$ stable quasimaps to $X$. It is a proper Deligne-Mumford stack and has perfect obstruction theory. Thus it has virtual cycle $[\overline{Q}_{1,0}(X,d)]^{\text{vir}}$ with zero virtual dimension. Let
$$
G_{1,0}:=\sum_{d=1}^{\infty}q^d\deg[\overline{Q}_{1,0}(X,d)]^{\text{vir}}.
$$

Let $\overline{Q}_{1,0}(\PP^{n-1},d)$ be the moduli space of genus one and degree $d$ stable quasimaps to $\PP^{n-1}$. It is a smooth Deligne-Mumford stack. Let $$\tilde\pi:\tilde\cC\to \overline{Q}_{1,0}(\PP^{n-1},d) $$ be the universal family, $\tilde\cS$ be the universal bundle over $\tilde\cC$.
Let $\tilde\iota$ be the closed immersion $\tilde\iota: \overline{Q}_{1,0}(X,d)\to \overline{Q}_{1,0}(\PP^{n-1},d)$. By the functoriality in \cite{KKP} we have
$$
\tilde\iota_{*}[\overline{Q}_{1,0}(X,d)]^{\text{vir}}=\mathbf{e}(\tilde\cV_1)\cap[\overline{Q}_{1,0}(\PP^{n-1},d)],
$$
where $\tilde\cV_1:=\oplus_{i=1}^{\ell} \tilde{\pi}_{*}\tilde\cS^{\vee\otimes a_i}$ is locally free. Thus
\beq\label{Cal}
\deg[\overline{Q}_{1,0}(X,d)]^{\text{vir}}=\deg(\mathbf{e}(\tilde\cV_1)\cap[\overline{Q}_{1,0}(\PP^{n-1},d)]).
\eeq

The standard torus $\TT=(\CC^*)^n$ action on $\PP^{n-1}$ induces an action on the moduli space $\overline{Q}_{1,0}(\PP^{n-1},d)$. There are two different types of the fixed loci when applying localization method to calculate the degree on the right hand side of (\ref{Cal}). One is the loop type and the other is the vertex type. To calculate the degree of the loop type, we use an infinitesimally marking to break the loop as Kim and Lho \cite{KB}, see Section 2.2. But in this paper we work out the loop contribution by the double $J$-functions given by Zinger \cite{Zin2}, which are directly related to the hypergeometric series used in calculation.

In the forthcoming papers, we will extend the calculation to the genus one stable quasimap invariants with markings, and genus two stable quasimap invariants for projective complete intersections.

Let $\cA_i(q)$ be as in Proposition \ref{equivred_prp_A}, $\cF^{(0,0)}(\alpha_i,q)$ and $\dot\Phi^{(0)}(\alpha_i,q)$ be as in Corollary \ref{expan}. By the localization method, we prove the following result.
\begin{theo}\label{main_2}For projective complete intersection Calabi-Yau $X$,
\beq\nonumber
q\frac{d}{d q}G_{1,0}=\bigg\{\sum_{i\in\mathbf{n}}\bigg(\cA_{i}(q)+\frac{1}{24}q\frac{d}{d q}\bigg(c_{i}(\alpha)\cF^{(0,0)}(\alpha_i,q)-\log \dot\Phi^{(0)}(\alpha_i,q)\bigg)\bigg)\bigg\}\bigg|_{\alpha=0},
\eeq
where $c_{i}(\alpha)=\sum_{k\neq i}\frac{1}{\alpha_k-\alpha_i}+\sum_{k=1}^{\ell}\frac{1}{a_k\alpha_i}$.
\end{theo}

Let $L(q)=(1-\a^{\a}q)^{-\frac{1}{n}}$ and $\mu(q)=\int_{0}^q\frac{L(u)-1}{u}du$, where $\a^{\a}=\prod_{k=1}^{\ell} a_k^{a_k}$. For $p\in\ZZ^{\ge0}$, let
$\dot{I}_p(q)\in\QQ[[q]]$ be defined as in Section 3. We can recover the following result of Kim and Lho from Theorem \ref{main_2}.
\begin{theo}\cite[Theorem 1.1]{KB}
For projective complete intersection Calabi-Yau $X$,

$$
G_{1,0}=\frac{1}{2}\AA(q)+\frac{1}{24}\bigg(\sum_{i=1}^{\ell}(\frac{n}{a_i}-\binom{n}{2})\mu(q)-\frac{n(\ell+1)}{2}\log L(q)\bigg),
$$
where
\begin{eqnarray*}
\AA(q)=\frac{n}{24}(n-1-2\sum_{r=1}^{\ell}\frac{1}{a_r})\mu(q)-\frac{3(n-1-\ell)^2+(n-2)}{24}\log(1-\a^{\a}q)\\
-\sum_{p=0}^{n-2-\ell}\binom{n-p-\ell}{2}\log \dot I_p(q).
\end{eqnarray*}
\end{theo}

{\bf Acknowledgement}: The author thanks Ionut Ciocan-Fontanine, Bumsig Kim and Hyenho Lho for useful discussion when the author visited KIAS. The author also thanks Huai-Liang Chang, Jun Li and Wanmin Liu for helpful discussions. This work was supported by Start-up Fund of Hunan University.

\section{Localization}

$\PP^{n-1}$ has a torus action induced by the standard torus $\TT=(\CC^*)^n$ action on $\CC^n$. The $\TT$ fixed points of $\PP^{n-1}$ are the 1-planes spanned by $e_i$, where $\{e_i\}_{i=1}^n$ is the standard basis of $\CC^n$. We label it by $p_{i}$. Let $\mathbf{n}$ be the set $1\le i\le n$.

Denote by $H^*(B\TT,\QQ):=\QQ[\alpha_1,\cdots,\alpha_n]$, where $\alpha_i\!=\!\pi_i^*c_1(\gamma)$. Here
$$\pi_i\!: (\PP^{\infty})^n\lra\PP^{\infty} \qquad\hbox{and}\qquad
\gamma^{\vee}\lra\PP^{\infty}$$
are the projection onto the $i$-th component and the tautological line bundle respectively. Denote by $\alpha=(\alpha_1,\cdots,\alpha_n)$. Let $\QQ_{\alpha}$ be the fractional field of $\QQ[\alpha_1,\cdots,\alpha_n]$.

Let $\mathbf{x}$ be the equivariant Chern class of the dual universal bundle. Then
$$
H^*_{\TT}(\PP^{n-1})\cong \QQ[\alpha_1,\cdots,\alpha_n][\mathbf{x}]/<\prod_{i=1}^n(\mathbf{x}-\alpha_i)>.
$$

 Let $$
\phi_{i}=\prod_{k\neq i}(\mathbf{x}-\alpha_k).
$$
Then $\phi_{i}$ is the equivariant Poincar$\acute{e}$ dual of the points $p_{i}\in H_{\TT}^*(\PP^{n-1})$. The Euler class

$$\mathbf{e}(T_{\PP^{n-1}})|_{p_{i}}=\prod_{k\neq i}(\alpha_i-\alpha_k)=\phi_{i}|_{p_{i}}$$

The Artiyah-Bott localization theorem states that
\beq
\int_{\PP^{n-1}}\eta=\sum\int_{p_{i}}\frac{\eta|_{p_{i}}}{\mathbf{e}(N_{p_{i}/\PP^{n-1}})},\quad\text{for all}\,\eta\in H^*_{\TT}(\PP^{n-1}).
\eeq
Therefore
\beq\label{AB}
\eta|_{p_{i}}=\int_{\PP^{n-1}}\eta \phi_{i}.
\eeq

So for $\eta\in H^*_{\TT}(\PP^{n-1})$
\beq\label{zero}
\eta|_{p_{i}}=0 \quad\text{for all}\quad i\in\mathbf{n}\quad\Longleftrightarrow\quad \eta=0.
\eeq

Denote by
$$\overline{Q}_{g,k|m}(\PP^{n-1},d)$$
the moduli space of genus $g$ degree $d$ quasimaps to $\PP^{n-1}$ with ordinary $k$ pointed markings and infinitesimally weighted $m$ pionted markings (see \cite[Section 2]{FK1}). When $g=1$ and $k=0,\,m=1$, the genus one moduli space $\overline{Q}_{1,0|1}(\PP^{n-1},d)$ is a smooth Deligne-Mumford stack since the obstruction vanishes. Let $$\pi:\cC\to \overline{Q}_{1,0|1}(\PP^{n-1},d) $$ be the universal family, $\cS$ be the universal bundle over $\cC$. Let $\cS^{\vee}$ be the dual bundle of $\cS$.

Let $\iota$ be the closed immersion $\iota: \overline{Q}_{1,0|1}(X,d)\to \overline{Q}_{1,0|1}(\PP^{n-1},d)$, by the functoriality in \cite{KKP} we have
$$
\iota_{*}[\overline{Q}_{1,0|1}(X,d)]^{\text{vir}}=\mathbf{e}(\cV_1)\cap[\overline{Q}_{1,0|1}(\PP^{n-1},d)],
$$
where $\cV_1:=\oplus_{i=1}^{\ell} \pi_{*}\cS^{\vee\otimes a_i}$ is locally free.

\subsection{Localization}

First we recall some facts about residues, which are from \cite[Section 1.2]{Zin1}

If $f$ is a rational function in $z$ and possibly other variables
and $z_0\!\in\!S^2$, let $\fR_{z=z_0}f(z)$ denote the
residue of the one-form $f(z)dz$ at $z\!=\!z_0$; thus,
$$\fR_{z=\infty}f(z)\equiv-\fR_{w=0}\big\{w^{-2}f(w^{-1})\big\}.$$
If $f$ involves variables other than $z$, $\fR_{z=z_0}f(z)$
is a function of the other variables. If $f$ is a power series in
$q$ with coefficients that are rational functions in~$z$ and
possibly other variables, let $\fR_{z=z_0}f(z)$ denote the
power series in~$q$ obtained by replacing each of the coefficients
by its residue at $z\!=\!z_0$. If $z_1,\ldots,z_k$ is a
collection of points in~$S^2$, not necessarily distinct, we define
$$\fR_{z=z_1,\ldots,z_k}f(z) \equiv
\sum_{y\in\{z_1,\ldots,z_k\}}\!\!\!\!\!\! \fR_{z=y}f(z).$$
%If $\hb_0\!\in\!\C$ or $\hb_0$ is one of the ``other'' variables in
%$f$, let
%$$\fR_{\hb=-\a\hb_0}f(\hb) \equiv \fR_{\hb=-a_1\hb_0,\ldots,-a_l\hb_0}f(\hb).$$
%For instance, if $\a=(2,2,3,3,3,3)$ and $\al_i$ is one of the other
%variables, then
%$$\fR_{\hb=-\a\al_i}f(\hb)\equiv\fR_{\hb=-2\al_i}f(\hb)+\fR_{\hb=-3\al_i}f(\hb).$$

%

Let $\hat\pi:\hat\cC\to\overline{Q}_{0,2}(\PP^{n-1},d)$ be the universal family, $\hat\cS$ be the universal bundle over $\hat\cC$.
Let
\beq\label{prod_e0}
\cV_{n;\a}^{(d)}:=\bigoplus_{k}R^0\hat\pi_*\big(\hat\cS^{\vee\otimes a_k}\big)
 \lra \overline{Q}_{0,2}(\PP^{n-1},d),\eeq
\beq\label{prod_e}
\dot\cV_{n;\a}^{(d)}:=\bigoplus_{k}R^0\hat\pi_*\big(\hat\cS^{\vee\otimes a_k}(-\si_1)\big)
 \lra \overline{Q}_{0,2}(\PP^{n-1},d)\eeq
 and
 \beq\ddot\cV_{n;\a}^{(d)}:=\bigoplus_{k}R^0\hat\pi_*\big(\hat\cS^{\vee\otimes a_k}(-\si_2)\big)
 \lra \overline{Q}_{0,2}(\PP^{n-1},d).
 \eeq
where $\sigma_i$ is the section correspondent to the marking.
Thus $\cV_{n;\a}^{(d)}$, $\dot\cV_{n;\a}^{(d)}$ and $\ddot\cV_{n;\a}^{(d)}$ are $\TT$-equivariant vector bundles.
Let
\beq\label{cZ2}
\dot{\cZ}(\mathbf{x},\hbar,q) :=
1+\sum_{d=1}^{\infty}q^d \ev_{1*}\!\!\left[\frac{\mathbf{e}(\dot\cV_{n;\a}^{(d)})}{\hbar\!-\!\psi_1}\right]
\in H_{\TT}^*(\PP^{n-1})[\hbar^{-1}]\big[\big[q\big]\big],\eeq

\beq\label{ZZdfn_e2}
\ddot{\cZ}(\mathbf{x},\hbar,q) :=
1+\sum_{d=1}^{\infty}q^d \ev_{1*}\!\!\left[\frac{\mathbf{e}(\ddot\cV_{n;\a}^{(d)})}{\hbar\!-\!\psi_1}\right]
\in H_{\TT}^*(\PP^{n-1})[\hbar^{-1}]\big[\big[q\big]\big].\eeq
 For $s\in\ZZ^{\ge0}$, Let

 \beq
\dot{\cZ}^{(s)}(\mathbf{x},\hbar,q) :=
\mathbf{x}^s+\sum_{d=1}^{\infty}q^d \ev_{1*}\!\!\left[\frac{\mathbf{e}(\dot\cV_{n;\a}^{(d)})\ev_2^*\mathbf{x}^s}{\hbar\!-\!\psi_1}\right]
\in H_{\TT}^*(\PP^{n-1})[\hbar^{-1}]\big[\big[q\big]\big],\eeq

\beq
\ddot{\cZ}^{(s)}(\mathbf{x},\hbar,q) :=
\mathbf{x}^s+\sum_{d=1}^{\infty}q^d \ev_{1*}\!\!\left[\frac{\mathbf{e}(\ddot\cV_{n;\a}^{(d)})\ev_2^*\mathbf{x}^s}{\hbar\!-\!\psi_1}\right]
\in H_{\TT}^*(\PP^{n-1})[\hbar^{-1}]\big[\big[q\big]\big].\eeq

 Let
\beq\label{ZZdfn2}
\dot{\cZ}(\hbar_1,\hbar_2,q) :=
\frac{[\overline{\Delta}]}{\hbar_1+\hbar_2}+\sum_{d=1}^{\infty}q^d\big\{\ev_1\!\times\!\ev_2\}_*\!\!\left[\frac{\mathbf{e}(\dot\cV_{n;\a}^{(d)})}
{(\hbar_1\!-\!\psi_1)(\hbar_2\!-\!\psi_2)}\right]
\eeq
$$
\in H_{\TT}^*(\PP^{n-1})[\hbar_1^{-1},\hbar_2^{-1}]\big[\big[q\big]\big],$$
where $[\overline{\Delta}]$ is the equivariant Poincar\'{e} dual of the diagonal class.
Denote by
\begin{eqnarray}\label{Z2ptdfn_e} &&\widetilde{\cZ}_{ij}^*(\hb_1,\hb_2,q) \\
&&=\frac{1}{2\hb_1\hb_2}(\dot{\cZ}(\hbar_1,\hbar_2,q)-\frac{[\overline{\Delta}]}{\hbar_1+\hbar_2})|_{p_{i}\times p_{j}}\nonumber\\
&&=\frac{1}{2\hb_1\hb_2}\sum_{d=1}^{\infty}q^d
\int_{\overline{Q}_{0,2}(\PP^{n-1},d)}
\frac{\mathbf{e}(\dot\cV_{n;\a}^{(d)})}{(\hb_1\!-\!\psi_1)(\hb_2\!-\!\psi_2)}
\ev_1^*\phi_{i}\ev_2^*\phi_{j}\nonumber.
\end{eqnarray}

Let
\begin{eqnarray*}
\cY(\mathbf{x},\hbar,q)=\sum_{d\ge0}q^d \cY_d(\hbar),\\
\dot{\cY}(\mathbf{x},\hbar,q)=\sum_{d\ge0}q^d \dot{\cY}_d(\hbar),\\
\ddot{\cY}(\mathbf{x},\hbar,q)=\sum_{d\ge0}q^d \ddot{\cY}_d(\hbar),
\end{eqnarray*}
where
\begin{eqnarray*}
\cY_d(\hbar)=\sum_{d}\frac{\prod_{k=1}^{\ell}\prod_{l=1}^{a_kd}(a_k\mathbf{x}+l\hbar)}{\prod_{l=1}^{d}
\prod_{k=1}^{n}(\mathbf{x}-\alpha_k+l\hbar)},\\
\dot{\cY}_d(\hbar)=\sum_{d}\frac{\prod_{k=1}^{\ell}\prod_{l=1}^{a_kd}(a_k\mathbf{x}+l\hbar)}{\prod_{l=1}^{d}\bigg(
\prod_{k=1}^{n}(\mathbf{x}-\alpha_k+l\hbar)-\prod_{k=1}^{n}(\mathbf{x}-\alpha_k)\bigg)},\\
\ddot{\cY}_d(\hbar)=\sum_{d}\frac{\prod_{k=1}^{\ell}\prod_{l=0}^{a_kd-1}a_k(\mathbf{x}+l\hbar)}{\prod_{l=1}^{d}
\bigg(\prod_{k=1}^{n}(\mathbf{x}-\alpha_k+l\hbar)-\prod_{k=1}^{n}(\mathbf{x}-\alpha_k)\bigg)}.
\end{eqnarray*}
Denote $\dot{I}_0(q)=\dot{\cY}(\mathbf{x},\hbar,q)|_{\mathbf{x}=\alpha=0\atop\hbar=1}$ and $\ddot{I}_0(q)=\ddot{\cY}(\mathbf{x},\hbar,q)|_{\mathbf{x}=\alpha=0\atop\hbar=1}$.
\begin{theo}\cite[Theorem 3]{Zin2}\label{thm1'}
If $\ell\!\in\!\ZZ^{\ge0}$, $n\!\in\!\ZZ^+$, and $\a\!\in\!(\ZZ^{>0})^{\ell}$ are such that
$|\a|:=\sum_{i=1}^{\ell}a_i=n$, then
\beq \dot{\cZ}(\mathbf{x},\hbar,q)=\frac{\dot{\cY}(\mathbf{x},\hbar,q)}{\dot{I}_0(q)}
\in H_{\TT}^*(\PP^{n-1})\big[\big[\hbar^{-1},q\big]\big],\eeq
and
\beq \ddot{\cZ}(\mathbf{x},\hbar,q)=\frac{\ddot{\cY}(\mathbf{x},\hbar,q)}{\ddot{I}_0(q)}
\in H_{\TT}^*(\PP^{n-1})\big[\big[\hbar^{-1},q\big]\big].\eeq
\end{theo}

%\begin{alignat}{1}
%\label{Z1ptdfn_e} \dot\cZ_{ij}^*(\hb,q) &=\dot\cZ(x,\hbar,q)|_{p_{ij}}-1=\sum_{d=1}^{\infty}q^d
%\int_{Q_{0,2}(\PP^{n-1},d)}\frac{\mathbf{e}(\dot\cV_{n;\a}^{(d)})}{\hb\!-\!\psi_1}\ev_1^*\phi_{ij};
%\end{alignat}

%Proposition below is obtained by applying the Atiyah-Bott
%Localization Theorem \cite{ABo} to the last expression in. As
%described in detail in \cite[Sections~1.3,~1.4]{bcov1}, the fixed
%loci of the $\TT$-action on $\cX\lpri$ are indexed by decorated
%graphs with one marked point. The vertices are decorated by elements
%of $[n]$, indicating the $\TT$-fixed point of $\PP^{n-1}$ to which
%the node or component corresponding to the vertex is mapped~to.

\subsection{Insertion of $0+$ weighted marking}
Let $$\hat{ev}_1:\overline{Q}_{1,0|1}(\PP^{n-1},d)\to [\CC^{n}/\CC^*]$$ be the evaluations map at the infinitesimally marking, see \cite[Section 2.3]{FK1}. Let
$$\tilde\gamma\in H^2_{\TT}([\CC^{n}/\CC^*])$$
be the lift of $\gamma\in H^2_{\TT}(\PP^{n-1})$, where $\gamma$ is the hyperplane class .

Let
$$
<\tilde\gamma>_{1,0|1,d}:=\int_{\overline{Q}_{1,0|1}(\PP^{n-1},d)}\mathbf{e}(\cV_1)\hat{ev}^*_1\tilde\gamma,
$$
then
$$
\sum_{d=1}^{\infty}q^d<\tilde\gamma>_{1,0|1,d}=q\frac{d}{d q}G_{1,0}.
$$
In the rest of this section we use localization method to calculate the left hand side of the above equation. As described in \cite[Section~7]{MOP},
the fixed loci of the $\TT$-action on $\overline{Q}_{1,0|1}(\PP^{n-1},d)$
are indexed by connected decorated graphs. Such a graph can be described by set $(\Ver,\Edg)$ of vertices, A decorated graph is a tuple
\beq\Gamma = \big(\Ver,\Edg;\mu,\mathfrak{d},\eta\big),\eeq
where $(\Ver,\Edg)$ is a graph as above and
$$\mu\!:\Ver\lra \mathbf{n}, \qquad \mathfrak{d}\!: \Ver\!\sqcup\!\Edg\lra\ZZ^{\ge0}\qquad\mbox{and}\qquad\eta:[1]\lra \Ver$$
are maps such that
\beq
\mu(v_1)\neq\mu(v_2)  ~~~\hbox{if}~~ \{v_1,v_2\}\in\Edg, \qquad
\mathfrak{d}(e)\neq0~~\forall\,e\!\in\!\Edg.\eeq

\noindent
Let
$$|\Gamma|\equiv\sum_{v\in\Ver}\!\!\mathfrak{d}(v)+\sum_{e\in\Edg}\!\!\mathfrak{d}(e)$$
be the degree of~$\Gamma$. Denoted by
$$\mbox{val}(v)\equiv (\big|\{e\!\in\!\Edg\!: v\!\in\!e\}\big|
, \big|\{i\!\in\![1]\!: \eta(1)\!=\!v\}\big|) .$$
for the vertices $v\!\in\!\Ver$.

Let $\cZ_{\Gamma}$ be the fixed locus of $\overline{Q}_{1,0|1}(\PP^{n-1},d)$ corresponds to a
decorated graph $\Gamma$. Thus
$$\cZ_{\Gamma}\approx \prod_{v\in\Ver}\!\!\overline{\cM}_{g_v,\val(v)|\mathfrak{d}(v)}$$
up to a finite group qoutient, where  $\overline{\cM}_{g_v,(k,m)|\mathfrak{d}(v)}$ denotes the moduli space weighted pointed stable curves with $k$ ordinary markings and $m$ infinitesimally markings. When $m=0$, we abbreviated by $\overline{\cM}_{g_v,k|\mathfrak{d}(v)}$. Let $$\pi_1:\cC_{\overline{\cM}_{g_v,\val(v)|\mathfrak{d}(v)}}\to \overline{\cM}_{g_v,\val(v)|\mathfrak{d}(v)}$$ be the restriction of the universal family.
%Denote by $$\dot{\cV}^{(|\mathfrak{d}(v_0)|)}_{n;\a}:=\oplus_{k}R^0\pi_{1*}\bigg(\cS^{\vee a_k}(-\sigma_1)\bigg),$$ where $\sigma_1$ is a given marking.

With $b_1,b_2,r\!\in\!\ZZ^{\ge0}$, let
\begin{eqnarray*}
&&\cF_{n}^{(b_1,b_2)}(\alpha_i,q)=\sum_{d=1}^{\infty}\frac{q^d}{d!}
\int_{\overline{\cM}_{0,2|d}}
\mathcal{G}_{n,d}^{(b_1,b_2)}(\alpha_i,q),
\end{eqnarray*}
where
\begin{eqnarray*}
\mathcal{G}_{n,d}^{(b_1,b_2)}(\alpha_i,q)&:=&\frac{\prod_{k\neq i}(\alpha_i-\alpha_k)\mathbf{e}(\dot\cV_{n;\a}^{(d)}(\alpha_i))\psi_1^{b_1}\psi_2^{b_2}}{\prod_{k\neq i} \mathbf{e}(R^0\pi_{1*}\cS^{\vee}(\alpha_i-\alpha_k))}.
\end{eqnarray*}

 By the proof of \cite[Theorem 4]{CZ}, we have
\beq\label{Fred_e}q\frac{d}{dq}\cF_{n}^{(b_1,b_2)}(\alpha_i,q)=\frac{1}{b_1!}\cF_{n}^{(0,0)}(\alpha_i,q)^{b_1}
\frac{1}{b_2!}\cF_{n}^{(0,0)}(\alpha_i,q)^{b_2}q\frac{d}{dq}\cF_{n}^{(0,0)}(\alpha_i,q).\eeq
By inductions on $b_2$, this gives
$$
\cF_{n}^{(0,b_2)}(\alpha_i,q)=\frac{1}{(b_2+1)!}\cF_{n}^{(0,0)}(\alpha_i,q)^{b_2+1}.
$$

Thus the $r=0$ case of \cite[Proposition 8.3]{CZ} is equivariant to
$$
\fR_{\hbar=0}\big\{\e^{-\frac{\cF_{n}^{(0,0)}(\alpha_i,q)}{\hbar}}\dot\cY(\alpha_i,\hbar,q)\big\}=0.
$$
Thus there is an expansion
\begin{coro}\cite[(4-9)]{Zin2}\label{expan}$$\e^{-\frac{\cF_{n}^{(0,0)}(\alpha_i,q)}{\hbar}}\dot\cY(\alpha_i,\hbar,q)=\sum_{m=0}^{\infty}\dot{\Phi}^{(m)}(\alpha_i,q)\hbar^m.$$
\end{coro}

In the case of moduli space $\overline{Q}_{1,0|1}(\PP^{n-1},d)$, they are two types of graph, either one distinguished vertex
 or one loop,
depending on whether the stable qusimaps they describe are constant
or not. The graphs with
one loop are called $A_{i}$-graphs.
In a graph of the $A_{i}$-type, the marked point~$1$ is attached to some vertex
$v_0\!\in\!\Ver$ that lies inside of the loop and is labeled~$i$.

Denote $\cA_{i}(q)$ by the total contribution from type $A_{i}$ graphs, then
\begin{prop}\label{equivred_prp_A}For every $i\in\mathbf{n}$,
\begin{eqnarray*}
\cA_{i}(q)&=&(\phi_{i}|_{p_{i}})^{-1}\bigg(q\frac{d}{dq}\cF_{n}^{(0,0)}(\alpha_i,q)+\alpha_i\bigg)\\
&&\fR_{\hb_1=0}\left\{\fR_{\hb_2=0}
\left\{\mbox{e}^{-\cF_{n}^{(0,0)}(\alpha_i,q)(\frac{1}{\hb_1}+\frac{1}{\hb_2})}\widetilde{\cZ}_{ii}^*(\hb_1,\hb_2,q)
 \right\}\right\}.
\end{eqnarray*}

\end{prop}

\begin{proof}
Let $\Gamma$ be a decorated graph of $A_{i}$ type, then we can break $\Gamma$ into a strand $\Gamma_{\pm}$ at $v_0=\eta(1)$, where $\pm$ are points attached to $v_0$.  Let $\mu(v_0)=i$. Thus $\cZ_{\Gamma}=\overline{\cM}_{0,(2,1)|\mathfrak{d}(v_0)}\times\cZ_{\Gamma_{\pm}}$, let
\beq
\pi_p:\cZ_{\Gamma}\to\overline{\cM}_{0,(2,1)|\mathfrak{d}(v_0)},\quad \mbox{and}\quad \pi_{e}:\cZ_{\Gamma}\to\cZ_{\Gamma_{\pm}}
\eeq
be the projections, then
\begin{eqnarray*}
\qquad \mathbf{e}(\cV_1)&=&\pi_p^*\mathbf{e}(\dot{\cV}^{(\mathfrak{d}(v_0))}_{n;\a})\pi_{e}^*\mathbf{e}(\dot{\cV}^{(d-\mathfrak{d}(v_0))}_{n;\a})\\
\frac{\mathbf{e}(T_{\mu(v_0)}\PP^{n-1})}{\mathbf{e}(\cN\cZ_{\Gamma})}&=&\frac{\prod_{k\neq i}(\alpha_i-\alpha_k)}{\prod_{k\neq i} \mathbf{e}(R^0\pi_{1*}\cS^{\vee}(\alpha_i-\alpha_k))}\\
&&\frac{\mathbf{e}(T_{\mu(v_0)}\PP^{n-1})\mathbf{e}(T_{\mu(v_0)}\PP^{n-1})}
{\mathbf{e}(\cN\cZ_{\Gamma_{\pm}})\,(\hb_{+}\!-\!\pi_e^*\psi_{+})\,(\hb_{-}\!-\!\pi_e^*\psi_{-})},
\end{eqnarray*}
where $\hb_{\pm}\!\equiv\!c_1(L_{\pm}')\!\in\!H^*(\overline{\cM}_{0,(2,1)|\mathfrak{d}(v_0)})$, $L_{\pm}'$ be the universal tangent line bundle at the marked point corresponding to $\pm$.
%Denote by
%$$
%\mathcal{G}_{ij}:=\frac{1}{\prod_{k\neq i,j} \mathbf{e}(R^0\pi_{1*}\sL_1(\alpha_i-\alpha_k))\prod_{k\neq i,j}\mathbf{e}(R^0\pi_{1*}\sL_2(\alpha_j-\alpha_k))}.
%$$

Let $e_{1}$ be the line $\{v_{1},v_0\}$, where $v_{1}$ is the nearest point to $+$. Let $e_{2}$ be the line $\{v_{2},v_0\}$, where $v_{2}$ is the nearest point to $-$. Let $\Gamma_e$ be the strand obtained by $\Gamma_{\pm}$ eliminating $e_{1},\,e_{2}$. Therefore we have
\begin{eqnarray}\label{A-type}
&&\sum_{d=1}^{\infty}q^d\int_{\cZ_{\Gamma}}\frac{\mathbf{e}(\cV_1)\hat{ev}_{1}^{*}\tilde\gamma}{\mathbf{e}(\cN\cZ_{\Gamma})}\\
&=&\sum_{d=1}^{\infty}\sum_{\mathfrak{d}(v_0)\ge0}\frac{q^{\mathfrak{d}(v_0)}}{\mathfrak{d}(v_0)!}\int_{\overline{\cM}_{0,(2,1)|\mathfrak{d}(v_0)}}\frac{\hat{ev}_{1}^*\tilde\gamma \,\mathbf{e}(\dot{\cV}^{(\mathfrak{d}(v_0))}_{n;\a}(\alpha_i))}{\prod_{k\neq i} \mathbf{e}(R^0\pi_{1*}\cS^{\vee}(\alpha_i-\alpha_k))}\nonumber\\
&& q^{d-\mathfrak{d}(v_0)}\int_{\cZ_{\Gamma_{\pm}}}\frac{\mathbf{e}(\dot{\cV}^{(d-\mathfrak{d}(v_0))}_{n;\a})(\alpha_i)\ev_1^*\phi_{i} \ev_2^*\phi_{i}}{\mathbf{e}(\cN\cZ_{\Gamma_{\pm}})\,(\hb_{+}\!-\!\pi_e^*\psi_{+})\,(\hb_{-}\!-\!\pi_e^*\psi_{-})}\nonumber\\
&=&\sum_{d=1}^{\infty}\sum_{\mathfrak{d}(v_0)\ge0}\frac{q^{\mathfrak{d}(v_0)}}{\mathfrak{d}(v_0)!}\int_{\overline{\cM}_{0,(2,1)|\mathfrak{d}(v_0)}}\frac{\hat{ev}_{1}^*\tilde\gamma \,\mathbf{e}(\dot{\cV}^{(\mathfrak{d}(v_0))}_{n;\a}(\alpha_i))\hb_{+}^{b_{1}}\hb_{-}^{b_{2}}}{\prod_{k\neq i} \mathbf{e}(R^0\pi_{1*}\cS^{\vee}(\alpha_i-\alpha_k))}\nonumber\\
&&q^{d-\mathfrak{d}(v_0)}\int_{\cZ_{\Gamma_{\pm}}}\frac{\mathbf{e}(\dot{\cV}^{(d-\mathfrak{d}(v_0))}_{n;\a}(\alpha_i))\ev_1^*\phi_{i} \ev_2^*\phi_{i}\pi_e^*\psi_{+}^{-(b_1+1)}\pi_e^*\psi_{-}^{-(b_2+1)}}{\mathbf{e}(\cN\cZ_{\Gamma_{\pm}})}\nonumber\\
&=&\sum_{d=1}^{\infty}\sum_{\mathfrak{d}(v_0)\ge0}\frac{q^{\mathfrak{d}(v_0)}}{\mathfrak{d}(v_0)!}\int_{\overline{\cM}_{0,(2,1)|\mathfrak{d}(v_0)}}\frac{\hat{ev}_{1}^*\tilde\gamma \,\mathbf{e}(\dot{\cV}^{(\mathfrak{d}(v_0))}_{n;\a}(\alpha_i))\hb_{+}^{b_{1}}\hb_{-}^{b_{2}}}{\prod_{k\neq i} \mathbf{e}(R^0\pi_{1*}\cS^{\vee}(\alpha_i-\alpha_k))}q^{d-\mathfrak{d}(v_0)}\nonumber\\
&&\dot{\fC}_{i}^{\mu(v_{1})}(\mathfrak{d}(e_{1}))\bigg(-\frac{\alpha_{\mu(v_{1})}-\alpha_i}{\mathfrak{d}(e_{1})}\bigg)^{-(b_1+1)}\ddot{\fC}_{i}^{\mu(v_{2})}
(\mathfrak{d}(e_{2}))\bigg(-\frac{\alpha_{\mu(v_{2})}-\alpha_{i}}{\mathfrak{d}(e_{2})}\bigg)^{-(b_{2}+1)}\nonumber\\
&&\int_{\cZ_{\Gamma_{e}}}\frac{\mathbf{e}(\dot{\cV}^{(d-\mathfrak{d}(v_0))-\mathfrak{d}(e_1)-\mathfrak{d}(e_2)}_{n;\a}(\alpha_i))\ev_1^*\phi_{\mu(v_{1})} \ev_2^*\phi_{\mu(v_{2})}}{\mathbf{e}(\cN\cZ_{\Gamma_{e}})(\hbar_1-\psi_1)(\hbar_2-\psi_2)}|_{\hbar_1=\frac{\alpha_{\mu(v_{1})}-\alpha_i}{\mathfrak{d}(e_{1})},
\hbar_2=\frac{\alpha_{\mu(v_{2})}-\alpha_i}{\mathfrak{d}(e_{2})}},\nonumber
\end{eqnarray}
where
$$
\dot{\fC}_{i}^{\mu(v_1)}(d)=\frac{\prod_{r=1}^{\ell}\prod_{l=1}^{a_{r}d}(a_{r}\alpha_i+\frac{l}{d}(\alpha_{\mu(v_1)}-\alpha_i))}
{d{\prod_{l=1}^{d}\prod_{m=1}^n\atop(l,m)\neq(d,k)}(\alpha_i-\alpha_m+\frac{l}{d}(\alpha_{\mu(v_1)}-\alpha_i))},
$$

$$
\ddot{\fC}_{i}^{\mu(v_2)}(d)=\frac{\prod_{r=1}^{\ell}\prod_{l=0}^{a_{r}d-1}(a_{r}\alpha_i+\frac{l}{d}(\alpha_{\mu(v_2)}-\alpha_i))}
{d{\prod_{l=1}^{d}\prod_{m=1}^n\atop(l,m)\neq(d,k)}(\alpha_{i}-\alpha_m+\frac{l}{d}(\alpha_{\mu(v_2)}-\alpha_i))}.
$$

Denote by
\beq
\overline{\cF}^{(b_1,b_2)}=\sum_{\mathfrak{d}(v_0)=0}^{\infty}\frac{q^{\mathfrak{d}(v_0)}}{\mathfrak{d}(v_0)!}\int_{\overline{\cM}_{0,(2,1)|\mathfrak{d}(v_0)}}\frac{\phi_{i}|_{p_i}\hat{ev}_{1}^*\tilde\gamma \,\mathbf{e}(\dot{\cV}^{(\mathfrak{d}(v_0))}_{n;\a}(\alpha_i))\psi_{1}^{b_{1}}\psi_{2}^{b_{2}}}{\prod_{k\neq i} \mathbf{e}(R^0\pi_{1*}\cS^{\vee}(\alpha_i-\alpha_k))}.
\eeq
%$$
%=|\mathfrak{d}(v_0)|\int_{\overline{\cM}_{0,2|\mathfrak{d}(v_0)}}\frac{\phi_{i} \,\mathbf{e}(\dot{\cV}^{(|\mathfrak{d}(v_0)|)}_{n;\a}(\alpha_i))\hb_{+}^{a_{1}}\hb_{-}^{a_{2}}}{\prod_{k\neq i} \mathbf{e}(R^0\pi_{1*}\cS^{\vee}(\alpha_i-\alpha_k))}.
%$$
Let $D_{1\hat{1},2}\subset \overline{\cM}_{0,(2,1)|\mathfrak{d}(v_0)}$ be the divisor whose general element is a two-component rational curve, with one component carrying the marked point 1 and $\hat{1}$ and the other carrying the marked point 2, where $\hat{1}$ means the infinitesimally marking. Then $\psi_2=D_{1\hat{1},2}$ on $\overline{\cM}_{0,(2,1)|\mathfrak{d}(v_0)}$, and
\begin{eqnarray*}
\overline{\cF}^{(b_1,b_2)}&=&\overline{\cF}^{(b_1,0)}\cF_{n}^{(0,b_2-1)}(\alpha_i,q)\\
&=&\overline{\cF}^{(0,0)}\cF_{n}^{(0,b_1-1)}(\alpha_i,q)\cF_{n}^{(0,b_2-1)}(\alpha_i,q).
\end{eqnarray*}
Because $\overline{\cM}_{0,(2,1)|\mathfrak{d}(v_0)}$ canonical isomorphic to the universal curve $$\cC_{\overline{\cM}_{0,2|\mathfrak{d}(v_0)}}\to \overline{\cM}_{0,2|\mathfrak{d}(v_0)},$$
we have
$$
\overline{\cF}^{(0,0)}=q\frac{d}{dq}\cF_{n}^{(0,0)}(\alpha_i,q)+\alpha_i.
$$
%When $|\mathfrak{d}(v_0)|=0$,
%\beq
%\int_{\overline{\cM}_{0,(2,1)|\mathfrak{d}(v_0)}}\frac{\phi_{i}\hat{ev}^*\tilde\gamma \,\mathbf{e}(\dot{\cV}^{(|\mathfrak{d}(v_0)|)}_{n;\a}(\alpha_i))\hb_{+}^{a_{1}}\hb_{-}^{a_{2}}}{\prod_{k\neq i} \mathbf{e}(R^0\pi_{1*}\cS^{\vee}(\alpha_i-\alpha_k))}
%\eeq
%\begin{eqnarray*}
%=\alpha_i&&.
%\end{eqnarray*}

By the recursion formula \cite[(7-12)]{Zin2} and the formulas in page 484 of \cite[Section 7]{Zin2}
\begin{eqnarray*}
&&\frac{1}{2}\sum_{\Gamma_{\pm}}q^{d-\mathfrak{d}(v_0)}\dot{\fC}_{i}^{\mu(v_1)}(\mathfrak{d}(e_{1}))\bigg(-\frac{\alpha_{\mu(v_{1})}-\alpha_i}{\mathfrak{d}(e_{1})}\bigg)^{-(b_1+1)} \ddot{\fC}_{i}^{\mu(v_2)}(\mathfrak{d}(e_{2}))\bigg(-\frac{\alpha_{\mu(v_{2})}-\alpha_{i}}{\mathfrak{d}(e_{2})}\bigg)^{-(b_{2}+1)}\\
&&
\int_{\cZ_{\Gamma_{e}}}\frac{\mathbf{e}(\dot{\cV}^{(d-\mathfrak{d}(v_0))-\mathfrak{d}(e_1)-\mathfrak{d}(e_2)}_{n;\a}(\alpha_i))\ev_1^*\phi_{\mu(v_{1})} \ev_2^*\phi_{\mu(v_{2})}}{\mathbf{e}(\cN\cZ_{\Gamma_{e}})(\hbar_1-\psi_1)(\hbar_2-\psi_2)}|_{\hbar_1=\frac{\alpha_{\mu(v_{1})}-\alpha_i}{\mathfrak{d}(e_{1})},
\hbar_2=\frac{\alpha_{\mu(v_{2})}-\alpha_i}{\mathfrak{d}(e_{2})}}
\end{eqnarray*}

$$
=\fR_{\hb_1=\frac{\alpha_{\mu(v_{1})}-\alpha_i}{\mathfrak{d}(e_{1})}}\left\{\fR_{\hb_2=\frac{\alpha_{\mu(v_{2})}-\alpha_{i}}{\mathfrak{d}(e_{2})}}
\left\{(-\hb_1)^{-b_1}(-\hb_2)^{-b_2}\widetilde{\cZ}_{ii}^*(\hb_1,\hb_2,q)
 \right\}\right\}.
$$
Then by the residue theorem on $S^2$, and the vanishing of the residue at $\infty$, which can be directly obtained from the expression of $\widetilde{\cZ}_{ii}^*(\hb_1,\hb_2,q)$. We have
\beq\label{A-type1}
\frac{1}{2}\sum_{\Gamma_{\pm}}q^{d-\mathfrak{d}(v_0)}\int_{\cZ_{\Gamma_{\pm}}}\frac{\mathbf{e}(\dot{\cV}^{(d-\mathfrak{d}(v_0))}_{n;\a}(\alpha_i))\ev_1^*\phi_{i} \ev_2^*\phi_{i}\pi_e^*\psi_{+}^{-(b_1+1)}\pi_e^*\psi_{-}^{-(b_2+1)}}{\mathbf{e}(\cN\cZ_{\Gamma_{\pm}})}
\eeq
$$
=\fR_{\hb_1=0}\left\{\fR_{\hb_2=0}
\left\{(-\hb_1)^{-b_1}(-\hb_2)^{-b_2}\widetilde{\cZ}_{ii}^*(\hb_1,\hb_2,q)
 \right\}\right\}.
$$
Combining formula (\ref{A-type}) and (\ref{A-type1}),
\begin{eqnarray*}
&&\frac{1}{2}\sum_{d=1}^{\infty}q^d\int_{\cZ_{\Gamma}}\frac{\mathbf{e}(\cV_1)\hat{ev}_{1}^{*}\tilde\gamma}{\mathbf{e}(\cN\cZ_{\Gamma})}\\
&=&(\phi_{i}|_{p_{i}})^{-1}\sum_{b_1\ge0,b_2\ge0}\bigg((q\frac{d}{dq}\cF_{n}^{(0,0)}(\alpha_i,q)+\alpha_i)
\frac{\cF_{n}^{(0,0)}(\alpha_i,q)^{b_1}}{b_1!}\frac{\cF_{n}^{(0,0)}(\alpha_i,q)^{b_2}}
{b_2!}\bigg)\\
&&\fR_{\hb_1=0}\left\{\fR_{\hb_2=0}
\left\{(-\hb_1)^{-b_1}(-\hb_2)^{-b_2}\widetilde{\cZ}_{ii}^*(\hb_1,\hb_2,q)
 \right\}\right\}.
\end{eqnarray*}

Thus
\begin{eqnarray*}
\cA_{i}&=&\frac{1}{2}\sum_{d=1}^{\infty}q^d\int_{\cZ_{\Gamma}}\frac{\mathbf{e}(\cV_1)\hat{ev}_{1}^{*}\tilde\gamma}{\mathbf{e}(\cN\cZ_{\Gamma})}\\
&=&(\phi_{i}|_{p_{i}})^{-1}\bigg(q\frac{d}{dq}\cF_{n}^{(0,0)}(\alpha_i,q)+\alpha_i\bigg)\\
&&\fR_{\hb_1=0}\left\{\fR_{\hb_2=0}
\left\{\mbox{e}^{-\cF_{n}^{(0,0)}(\alpha_i,q)(\frac{1}{\hb_1}+\frac{1}{\hb_2})}\widetilde{\cZ}_{ii}^*(\hb_1,\hb_2,q)
 \right\}\right\}.
\end{eqnarray*}

\end{proof}

When the domain curve is mapped to the fixed point $p_{i}$, then the correpondent decorated graph $\Gamma$ is a vertex. They will be called $B_{i}$-types. In a graph of the $B_{i}$-type, the infinitesimally marked point~$1$ is attached to a vertex labeled~$i$.

Let $\cB_{i}(q)$ be the total contribution from type $B_{i}$ graphs. Then

\begin{prop}\label{equivred_prp_B}
For every $i\in\mathbf{n}$,
\begin{eqnarray*}
\cB_{i}(q)&=&\frac{1}{24}q\frac{d}{d q}\bigg(c_{i}(\alpha)\cF^{(0,0)}(\alpha_i,q)-\log \dot\Phi^{(0)}(\alpha_i,q)\bigg),
\end{eqnarray*}
where  $c_{i}(\alpha)=\sum_{k\neq i}\frac{1}{\alpha_k-\alpha_i}+\sum_{k=1}^{\ell}\frac{1}{a_k\alpha_i}$.
\end{prop}
\begin{proof}
The proof is exact the same as the calculation of the vertex contribution in \cite{KB}. We sketch it as following.
Let $\Gamma$ be a type $B_{i}$ decorated graph, it is just vertex $v$ over $p_{i}$. Thus $\cZ_{\Gamma}=\overline{\cM}_{1,(0,1)|d}$. The contribution

\beq
\cB_{i}(q)=\sum_{d=1}^{\infty}\frac{q^d}{d!}
\int_{\overline{\cM}_{1,(0,1)|d}}\frac{\hat{ev}_{1}^{*}\tilde\gamma\,\mathbf{e}(\mathbb{E} ^\vee \otimes T_{p_{i}}\PP^{n-1})}{\mathbf{e} (T_{p_{i}}\PP^{n-1})}
\mathbf{e}(\cV_1(\alpha_i))\mathrm{Q}_{v}
\eeq
$$
=q\frac{d}{d q}\bigg(\sum_{d=1}^{\infty}\frac{q^d}{d!}
\int_{\overline{\cM}_{1,0|d}}\frac{\mathbf{e}(\mathbb{E} ^\vee \otimes T_{p_{i}}\PP^{n-1})}{\mathbf{e} (T_{p_{i}}\PP^{n-1})}
\mathbf{e}(\cV_1(\alpha_i))\mathrm{Q}_{v}\bigg)
$$
where $\mathbb{E}$ is the Hodge bundle, \beq\label{Q-def}
\mathrm{Q}_{v}=\frac{1}{\prod\limits_{k\neq
i}\!\!\mathbf{e}(H^0(C_{v},\sO_{C_{v}}(D_1)|_{D_1})(\alpha_i\!-\!\alpha_k))},\eeq $\sO_{C_{v}}(D_1)\cong\cS^{\vee}|_{C_{v}}$.

$$
\mathbf{e}(\cV_1(\alpha_i))=\frac{\mathbf{e}(E|_{p_{i}})}{\mathbf{e}(\mathbb{E}^{\vee}\otimes E|_{p_{i}} )}\widetilde{Q},
$$
where

\beq\label{Q-def}
\widetilde{Q}=\frac{1}{\prod_{i=1}^{\ell}\mathbf{e}(H^0(C_{v},\sO_{C_{v}}(\widetilde{D}_i)|_{\widetilde{D}_i})(\alpha_i))},\eeq
$\sO_{C_{v}}(\widetilde{D}_i)\cong\cS^{\vee\otimes a_i}|_{C_{v}}$.

Let $c_{i}(\alpha)$ determined by
$$
1+c_{i}(\alpha)\mathbf{e}(\EE)=\frac{\mathbf{e}(\mathbb{E} ^\vee \otimes T_{p_{i}}\PP^{n-1})}{\mathbf{e} (T_{p_{i}}\PP^{n-1})}\frac{\mathbf{e}(E|_{p_{i}})}{\mathbf{e}(\mathbb{E}^{\vee}\otimes E|_{p_{i}} )}.
$$
Thus
$$c_{i}(\alpha)=\sum_{k\neq i}\frac{1}{\alpha_k-\alpha_i}+\sum_{k=1}^{\ell}\frac{1}{a_k\alpha_i}.
$$
Denote by $F^{(1,0)}_{d}=Q_v\,\widetilde{Q}$, then
\begin{eqnarray*}
&&\int_{\overline{\cM}_{1,0|d}}\frac{\mathbf{e}(\mathbb{E} ^\vee \otimes T_{p_{i}}\PP^{n-1})}{\mathbf{e} (T_{p_{i}}\PP^{n-1})}
\mathbf{e}(\cV_1(\alpha_i))\mathrm{Q}_{v})\\
&&=\int_{\overline{\cM}_{1,0|d}}(1+c_{i}(\alpha)\mathbf{e}(\EE))F^{(1,0)}_{d}.
\end{eqnarray*}

For nonnegative integers $g$ and $m$, the above expression for $F^{(1,0)}_d$ also defined as an element in $H^*(\overline{M}_{g,m|d},\QQ_{\alpha})$, which can be written as a polynomial of diagonal classes and the psi classes. By the proof in \cite[Theorem 2.6]{KB}, we have
$$
\int_{\overline{\cM}_{1,0|d}}\mathbf{e}(\EE)\,F^{(1,0)}_{d}=\frac{1}{24}\int_{\overline{\cM}_{0,2|d}}F^{(0,2)}_{d},
$$
$$
\sum_{d=1}^{\infty}\frac{q^d}{d!}\int_{\overline{\cM}_{1,0|d}}F^{(1,0)}_{d}
=\frac{1}{24}\log(\sum_{d=0}^{\infty}\frac{q^d}{d!}\int_{\overline{\cM}_{0,3|d}}F^{(0,3)}_{d}).
$$
Therefore
\begin{eqnarray}\label{part1}
&& \sum_{d=1}^{\infty}\frac{q^d}{d!}
\int_{\overline{\cM}_{1,0|d}}\mathbf{e}(\EE)\,F^{(1,0)}_{d}\\
&=&\frac{1}{24}\sum_{d=1}^{\infty}\frac{q^d}{d!}
 \int_{\overline{\cM}_{0,(2,0)|d}}F^{(0,2)}_{d}\nonumber\\
&=&\frac{1}{24}\sum_{d=1}^{\infty}\frac{q^d}{d!}
 \int_{\overline{\cM}_{0,(2,0)|d}}\frac{\prod_{k\neq i}(\alpha_i-\alpha_k)\mathbf{e}(\dot\cV_{n;\a}^{(d)}(\alpha_i))}{\prod_{k\neq i} \mathbf{e}(R^0\pi_{1*}\cS^{\vee}(\alpha_i-\alpha_k))}\nonumber\\
&=&\frac{1}{24}\cF^{(0,0)}(\alpha_i,q).\nonumber
\end{eqnarray}
By \cite[Proposition 4.1]{Zin2}, we have
\beq\label{part2}
\sum_{d=0}^{\infty}\frac{q^d}{d!}\int_{\overline{\cM}_{0,3|d}}F^{(0,3)}_{d}=\frac{1}{\dot\Phi^{(0)}(\alpha_i,q)}.
\eeq
Thus combining (\ref{part1}) and (\ref{part2}), we have

\begin{eqnarray*}
\cB_{i}(q)&=&\frac{1}{24}q\frac{d}{d q}\bigg(c_{i}(\alpha)\cF^{(0,0)}(\alpha_i,q)-\log \dot\Phi^{(0)}(\alpha_i,q)\bigg).
\end{eqnarray*}
\end{proof}
Combining Proposition \ref{equivred_prp_A} and Proposition \ref{equivred_prp_B}, we have
\begin{theo} \label{The}For Calabi-Yau manifold $X\subset \PP^{n-1}$ which is complete intersection,

\beq\label{form}
q\frac{d}{d q}G_{1,0}=\bigg\{\sum_{i\in\mathbf{n}}\bigg(\cA_{i}(q)+\frac{1}{24}q\frac{d}{d q}\bigg(c_{i}(\alpha)\cF^{(0,0)}(\alpha_i,q)-\log \dot\Phi^{(0)}(\alpha_i,q)\bigg)\bigg)\bigg\}\bigg|_{\alpha=0},
\eeq
where $c_{i}(\alpha)=\sum_{k\neq i}\frac{1}{\alpha_k-\alpha_i}+\sum_{k=1}^{\ell}\frac{1}{a_k\alpha_i}$.
\end{theo}

\section{Calculation}
In this section we work out the explicit expression of (\ref{form}) by using the hypergeometric series and their properties established in \cite{Popa}, \cite{Zin1} and \cite{Zin2}.

We define power series $L_{n},\xi_{n}\in\QQ_{\alpha}[\mathbf{x}][[q]]$ by
\begin{alignat}{2}
L_{n}&\in \mathbf{x}+q\QQ_{\alpha}[\mathbf{x}][[q]], &\qquad
\tilde{\mathbf{s}}_n\big(L_{n}(\mathbf{x},q)\big)-q\a^{\a}L_{n}(\mathbf{x},q)^{|\a|}&=\tilde{\mathbf{s}}_n(\mathbf{x}), \\
\xi_{n}&\in q\QQ_{\alpha}[\mathbf{x}][[q]],&\qquad
\mathbf{x}+q\frac{d}{d q}\xi_{n}(\mathbf{x},q)&=L_{n}(\mathbf{x},q),\notag
\end{alignat}
where $\tilde{\mathbf{s}}_r(y)$ is the r-th elementary symmetric polynomial in $\{y-\alpha_k\}$. By \cite[(4-9)]{Zin2}, $\xi_{n}(\alpha_i,q)=\cF_{n}^{(0,0)}(\alpha_i,q)$. Let $$L(q)=(1-\a^{\a}q)^{-\frac{1}{n}},\quad\quad\mu(q)=\int_{0}^q\frac{L(u)-1}{u}du,$$ then

 \beq\label{L-for}
 L_{n}(\mathbf{x},q)=L(q)\mathbf{x}+\sum_{d=0}^{\infty}f_d(\mathbf{x},\alpha)q^d,
 \eeq
 where $f_d(\mathbf{x},\alpha)\in \QQ_{\alpha}[\mathbf{x}]$ with $\mathbf{x}|f_d(\mathbf{x},\alpha)$, and has no $\mathbf{x}$ term with constant coefficient,  $f_d(\mathbf{x},0)=0$.

 \beq\label{xi-for}
 \xi_{n}(\mathbf{x},q)=\mu(q)\mathbf{x}+\sum_{d=1}^{\infty}g_d(\mathbf{x},\alpha)q^d,
 \eeq
 where $g_d(\mathbf{x},\alpha)\in \QQ_{\alpha}[\mathbf{x}]$  with $\mathbf{x}|g_d(\mathbf{x},\alpha)$, and has no $\mathbf{x}$ term with constant coefficient, $g_d(\mathbf{x},0)=0$.

By the residue theorem on $S^2$,
\begin{eqnarray*}
\sum_{i=1}^n\sum_{k\neq i}\frac{\xi_{n}(\alpha_i,q)}{\alpha_k-\alpha_i}&=&-\sum_{i}\fR_{z=\alpha_i}\sum_{k\neq j}\frac{\mu(q)z}{(z-\alpha_j)(z-\alpha_k)}\\
&&-\sum_{d=0}^{\infty}q^d\sum_{i}\fR_{z=\alpha_i}\sum_{k\neq j}\frac{g_d(z,\alpha)}{(z-\alpha_j)(z-\alpha_k)}\\
&=&\fR_{z=\infty}\sum_{k\neq j}\frac{\mu(q)z}{(z-\alpha_j)(z-\alpha_k)}\\
&&+\sum_{d=0}^{\infty}q^d\fR_{z=\infty}\sum_{k\neq j}\frac{g_d(z,\alpha)}{(z-\alpha_j)(z-\alpha_k)}\\
&=&-\binom{n}{2}\mu(q)+\sum_{d=0}^{\infty}q^d\fR_{z=\infty}\sum_{k\neq j}\frac{g_d(z,\alpha)}{(z-\alpha_j)(z-\alpha_k)}.
\end{eqnarray*}
By (\ref{xi-for}),
$$
\bigg(\fR_{z=\infty}\sum_{k\neq j}\frac{g_d(z,\alpha)}{(z-\alpha_j)(z-\alpha_k)}\bigg)\bigg|_{\alpha=0}=0.
$$
Thus
$$
\bigg(\sum_{i=1}^n\sum_{k\neq i}\frac{\xi_{n}(\alpha_i,q)}{\alpha_k-\alpha_i}\bigg)\bigg|_{\alpha=0}=-\binom{n}{2}\mu(q).
$$
\begin{eqnarray*}
\sum_{i=1}^{n}\frac{\xi_{n}(\alpha_i,q)}{a_k\alpha_i}&=&\frac{1}{a_k}\sum_{i=1}^n\fR_{z=\alpha_i}
\sum_{j=1}^{n}\frac{\xi_{n}(z,q)}{z(z-\alpha_j)}\\
&=&\frac{1}{a_k}\sum_{i=1}^n\fR_{z=\alpha_i}
\sum_{j=1}^{n}\frac{\mu(q)}{(z-\alpha_j)}\\
&&+\frac{1}{a_k}\sum_{d=0}^{\infty}q^d\sum_{i=1}^n\fR_{z=\alpha_i}
\sum_{j=1}^{n}\frac{g_d(z,\alpha)}{z(z-\alpha_j)}\\
&=&-\frac{1}{a_k}\fR_{z=\infty}
\sum_{j=1}^{n}\frac{\mu(q)}{(z-\alpha_j)}\\
&&-\frac{1}{a_k}\sum_{d=0}^{\infty}q^d\sum_{i=1}^n\fR_{z=\infty}
\sum_{j=1}^{n}\frac{g_d(z,\alpha)}{z(z-\alpha_j)}\\
&&-\frac{1}{a_k}\sum_{d=0}^{\infty}q^d\sum_{i=1}^n\fR_{z=0}
\sum_{j=1}^{n}\frac{g_d(z,\alpha)}{z(z-\alpha_j)},
\end{eqnarray*}
by (\ref{xi-for})
\begin{eqnarray*}
\frac{1}{a_k}\sum_{d=0}^{\infty}q^d\bigg(\sum_{i=1}^n\fR_{z=0,\infty}
\sum_{j=1}^{n}\frac{g_d(z,\alpha)}{z(z-\alpha_j)}\bigg)\bigg|_{\alpha=0}=0.
\end{eqnarray*}
Thus
$$
\bigg(\sum_{k=1}^{\ell}\sum_{i=1}^{n}\frac{\xi_{n}(\alpha_i,q)}{a_k\alpha_i}\bigg)\bigg|_{\alpha=0}
=\sum_{k=1}^{\ell}\frac{n}{a_k}.
$$
 By \cite[(4-10)]{Zin2}

$$
\dot\Phi^{(0)}(\alpha_i,q)|_{\alpha=0}=L(q)^{\frac{\ell+1}{2}}.
$$
Therefore
\begin{eqnarray}\label{B-va}
&&\bigg(\sum_{i=1}^n\cB_{i}(q)\bigg)\bigg|_{\alpha=0}\\
&=&\frac{1}{24}\bigg(\sum_{i=1}^n\bigg(c_{i}(\alpha)\cF^{(0,0)}(\alpha_i,q)-\log \dot\Phi^{(0)}(\alpha_i,q)\bigg)\bigg)\bigg|_{\alpha=0}\nonumber\\
&=&\frac{1}{24}\bigg(\sum_{i=1}^{\ell}(\frac{n}{a_i}-\binom{n}{2})\mu(q)-\frac{n(\ell+1)}{2}\log L(q)\bigg).\nonumber
\end{eqnarray}

For each $p\in\mathbf{n}$, let $\sigma_p$ be the $p$-th elementary symmetric
polynomial in $\alpha_1,\ldots,\alpha_n$.
Denote by
$$\QQ[\alpha]^{S_n}\equiv\QQ[\alpha_1,\ldots,\alpha_n]^{S_n}\subset\QQ [\alpha_1,\ldots,\alpha_n]$$
the subspace of symmetric polynomials, by $\mathcal{J}\subset\QQ[\alpha]^{S_n}$
the ideal generated by $\sigma_1,\ldots,\sigma_{n-1}$, and~by
$$\tilde\QQ[\alpha]^{S_n} \equiv
\QQ[\alpha_1,\ldots,\alpha_n]_{<(\alpha_j-\alpha_k)|j\neq k>}^{S_n} \subset \QQ_{\alpha}$$
the subalgebra of symmetric rational functions in $\alpha_1,\ldots,\alpha_n$
whose denominators are products of $(\alpha_j\!-\!\alpha_k)$ with $j\!\neq\!k$.
For each $i\!=\!1,\ldots,n$, let
$$\tilde\QQ_i[\alpha]^{S_{n-1}} \equiv
\QQ[\alpha_1,\ldots,\alpha_n]_{<(\alpha_i-\alpha_k)|k\neq i>}^{S_{n-1}}
\subset \QQ_{\alpha}$$
be the subalgebra consisting of rational functions symmetric in $\{\alpha_k\!:k\!\neq\!i\}$
and with denominators that are  products of
$(\alpha_i\!-\!\alpha_k)$ with $k\!\neq\!i$.

%For each $i\!=\!1,\ldots,n$, let
%$$\tilde\QQ_i[\alpha]^{S_{n-1}}_{\hbar}\equiv
%\QQ(\hbar,\alpha_i)[\alpha]_{<(\alpha_i-\alpha_k+r\hbar)|k\in i,r\in\ZZ,k\neq i>}^{S_{n-1}}
%\subset\QQ_{\alpha}(\hbar)$$
%be the subalgebra consisting of rational functions symmetric in $\{\alpha_k\!:k\!\neq\!i\}$
%and with denominators that are a product of a polynomial with rational coefficients in $\hbar$ and $\alpha_i$ and
%of linear factors of the form $(\alpha_i\!-\!\alpha_k\!+\!r\hbar),r\in\ZZ$.
%We define
%$$\tilde\QQ_i[\alpha]^{S_{n-1}}_{\hb_1,\hb_2}\subset \QQ_{\alpha}(\hbar_1,\hbar_2)$$
%to be the subalgebra generated by $\tilde\QQ_i[\alpha]^{S_{n-1}}_{\hb_1}$  and $\tilde\QQ_i[\alpha]^{S_{n-1}}_{\hb_2}$.
%%\begin{lemm}\cite[Lemma 5.1]{Popa}
%%$$\mbox{Span}\{\alpha_i^{n-2}\}\cap \mathcal{J}\tilde\QQ_i[\alpha]^{S_{n-1}}_{\hb_1,\hb_2}=0$$
%%\end{lemm}
\begin{lemm}\label{good0}
Let $f(z,\alpha)\in \mathcal{J}\QQ[\alpha][z]$. Then
for $\frac{f(\alpha_j,\alpha)}{\prod_{k\neq j}(\alpha_j-\alpha_k)^m}$ with $m\ge0$, we have
\begin{eqnarray}
&&\bigg(\sum_{j=1}^n\frac{f(\alpha_j, \alpha)}{\prod_{k\neq j}(\alpha_j-\alpha_k)^{m+1}}\bigg)\bigg|_{\alpha=0}=0.
\end{eqnarray}
\end{lemm}
\begin{proof}

By the residue theorem on $S^2$,
when $m=0$,
\begin{eqnarray*}
&&\bigg(\sum_{j=1}^n\frac{f(\alpha_j,\alpha)}{\prod_{k\neq j}(\alpha_j-\alpha_k)}\bigg)\bigg|_{\alpha=0}\\
&=&\bigg(\sum_{j=1}^n\fR_{z=\alpha_j}
\bigg\{\frac{  f(z,\alpha)}{\prod_{k=1}^n(z-\alpha_k)}\bigg\}\bigg)\bigg|_{\alpha=0}\nonumber\\
&=&-\bigg(\fR_{\infty}
\bigg\{\frac{ f(z,\alpha)}{\prod_{k=1}^n(z-\alpha_k)}\bigg\}\bigg)\bigg|_{\alpha=0}\nonumber\\
&=&0.\nonumber
\end{eqnarray*}

When $m\ge1$,
\begin{eqnarray*}
&&\bigg(\sum_{j=1}^n\frac{f(\alpha_j,\alpha)}{\prod_{k\neq j}(\alpha_j-\alpha_k)^{m+1}}\bigg)\bigg|_{\alpha=0}\\
&=&\bigg(\sum_{j=1}^n\fR_{z=\alpha_j}
\bigg\{\sum_{i=1}^n\frac{ (z-\alpha_i)^m f(z,\alpha)}{\prod_{k=1}^n(z-\alpha_k)^{m+1}}\bigg\}\bigg)\bigg|_{\alpha=0}\nonumber\\
&=&-\bigg(\fR_{\infty}
\bigg\{\sum_{i=1}^n\frac{ (z-\alpha_i)^mf(z,\alpha)}{\prod_{k=1}^n(z-\alpha_k)^{m+1}}\bigg\}\bigg)\bigg|_{\alpha=0}\nonumber\\
&=&0.\nonumber
\end{eqnarray*}
\end{proof}

%\begin{lemm}\label{good0}
%Let $f(z,\alpha)\in Q[\alpha][z][[q]]$ such that $f(\alpha_i,\alpha)\in \mathcal{J}Q[\alpha][[q]]$. Then
%for $\frac{f(\alpha_j,\alpha)}{\alpha_j^{m'}\prod_{k\neq j}(\alpha_j-\alpha_k)^m}$ with $m',m\ge0$
%\begin{eqnarray}
%&&\bigg(\sum_{j=1}^n\frac{f(\alpha_j, \alpha)}{\alpha_j^{m'}\prod_{k\neq j}(\alpha_j-\alpha_k)^{m+1}}\bigg)\bigg|_{\alpha=0}=0.
%\end{eqnarray}
%\end{lemm}
%\begin{proof}
%
%By the residue theorem on $S^2$,
%\begin{eqnarray*}
%&&\bigg(\sum_{j=1}^n\frac{f(\alpha_j,\alpha)}{\alpha_j^{m'}\prod_{k\neq j}(\alpha_j-\alpha_k)^{m+1}}\bigg)\bigg|_{\alpha=0}\\
%&=&\bigg(\sum_{j=1}^n\fR_{z=\alpha_j}
%\bigg\{\sum_{i=1}^n\frac{ (z-\alpha_i)^m f(z,\alpha)}{z^{m'}\prod_{k=1}^n(z-\alpha_k)^{m+1}}\bigg\}\bigg)\bigg|_{\alpha=0}\nonumber\\
%&=&-\bigg(\fR_{0,\infty}
%\bigg\{\sum_{i=1}^n\frac{ (z-\alpha_i)^mf(z,\alpha)}{z^{m'}\prod_{k=1}^n(z-\alpha_k)^{m+1}}\bigg\}\bigg)\bigg|_{\alpha=0}\nonumber\\
%&=&0.\nonumber
%\end{eqnarray*}
%\end{proof}

\begin{lemm}\label{good'0}
Let $d,m\in\ZZ^{\ge0}$ satisfy, $ m\neq (n-1)(d+1)$. Then

\begin{eqnarray}
&&\bigg(\sum_{j=1}^n\frac{\alpha_j^{m}}{\prod_{k\neq j}(\alpha_j-\alpha_k)^{d+1}}\bigg)\bigg|_{\alpha=0}=0.
\end{eqnarray}
\end{lemm}
\begin{proof}

By the residue theorem on $S^2$, when $d=0$
\begin{eqnarray*}
&&\bigg(\sum_{j=1}^n\frac{\alpha_j^{m}}{\prod_{k\neq j}(\alpha_j-\alpha_k)}\bigg)\bigg|_{\alpha=0}\\
&=&\bigg(\sum_{j=1}^n\fR_{z=\alpha_j}
\bigg\{\frac{ z^{m}}{\prod_{k=1}^n(z-\alpha_k)}\bigg\}\bigg)\bigg|_{\alpha=0}\nonumber\\
&=&-\bigg(\fR_{\infty}
\bigg\{\frac{ z^{m}}{\prod_{k=1}^n(z-\alpha_k)}\bigg\}\bigg)\bigg|_{\alpha=0}\nonumber\\
&=&0.\nonumber
\end{eqnarray*}
When $d\ge 1$,
\begin{eqnarray*}
&&\bigg(\sum_{j=1}^n\frac{\alpha_j^{m}}{\prod_{k\neq j}(\alpha_j-\alpha_k)^{d+1}}\bigg)\bigg|_{\alpha=0}\\
&=&\bigg(\sum_{j=1}^n\fR_{z=\alpha_j}
\bigg\{\sum_{i=1}^n\frac{  (z-\alpha_i)^dz^m}{\prod_{k=1}^n(z-\alpha_k)^{d+1}}\bigg\}\bigg)\bigg|_{\alpha=0}\nonumber\\
&=&-\bigg(\fR_{\infty}
\bigg\{\sum_{i=1}^n\frac{  (z-\alpha_i)^dz^m}{\prod_{k=1}^n(z-\alpha_k)^{d+1}}\bigg\}\bigg)\bigg|_{\alpha=0}\nonumber\\
&=&0.\nonumber
\end{eqnarray*}
\end{proof}

Let
\beq F(w,q)\equiv\sum_{d=0}^{\infty}q^d \,
\frac{\prod_{k=1}^{\ell}\prod\limits_{r=1}^{a_kd}(a_kw\!+\!r)}
{\prod\limits_{r=1}^{d}(w\!+\!r)^n} \in \QQ(w)\big[\big[q\big]\big],\eeq

\beq\dot{F}(w,q)\equiv\sum_{d=0}^{\infty}q^d \,
\frac{\prod_{k=1}^{\ell}\prod\limits_{r=1}^{a_kd}(a_kw\!+\!r)}
{\prod\limits_{r=1}^{d}\big((w\!+\!r)^n-w^n\big)} \in \QQ(w)\big[\big[q\big]\big],\eeq

 and
\beq\ddot{F}(w,q)\equiv\sum_{d=0}^{\infty}q^d \,
\frac{\prod_{k=1}^{\ell}\prod\limits_{r=0}^{a_kd-1}(a_kw\!+\!r)}
{\prod\limits_{r=1}^{d}\big((w\!+\!r)^n-w^n\big)} \in \QQ(w)\big[\big[q\big]\big].\eeq

These are power series in $q$ with constant term 1 whose coefficients are rational functions
in $w$ which are regular at~$w=0$.
We denote the subgroup of all such power series by~$\cP$
and define
\beq
\begin{aligned}
&\bD\!:\QQ(w)\big[\big[q\big]\big]\lra \QQ(w)\big[\big[q\big]\big],
&\quad& \bM:\cP\lra\cP  \qquad\hbox{by}\\
&\bD H(w,q)\equiv \left\{1+\frac{q}{w}\frac{d}{d q}\right\}H(w,q),
&\quad&
\bM H(w,q)\equiv\bD\bigg(\frac{H(w,q)}{H(0,q)}\bigg).
\end{aligned}\eeq
For $s\in\ZZ^{\ge0}$, let
\beq \dot{I}_s(q)\equiv \bM^s\dot{F}(0,q), \qquad
\ddot{I}_s(q)\equiv \bM^s\ddot{F}(0,q).\eeq
By \cite[(4.14)]{Popa}, for $p\ge0$, $$
\bM^{p}\dot{F}(w,q)=\bM^{p+\ell}\ddot{F}(w,q).
$$

Let
$$\fD^0\check{\cY}(\mathbf{x},\hb,q)=\frac{\check{\cY}(\mathbf{x},\hb,q)}{\check{I}_0(q)}, \quad
\fD^s\check{\cY}(\mathbf{x},\hb,q)=
\frac{1}{\check{I}_s(q)}
\left\{\mathbf{x}+\hb\, q\frac{d}{d q}\right\}\fD^{s-1}\check{\cY}(\mathbf{x},\hb,q)$$
for all $s\!\in\!\ZZ^+$ and
$(\check{\cY},\check{I})\!=\!(\cY,I), (\dot\cY,\dot{I}), (\ddot\cY,\ddot{I})$.
For $r,s,s'\ge 0$, there exists $\cC_{s,s'}^{(r)}\in\QQ[\alpha_1,\cdots,\alpha_n][[q]]$, such that
\beq\label{good3}
\hbar^s\sum_{s'=0}^{\infty}\sum_{r=0}^{s'}\cC_{s,s'}^{(r)}(q)\mathbf{x}^{s'-r}\hbar^{-s'}
=\fD^s\cY(\mathbf{x},\hbar,q).\eeq
\begin{gather}\label{good4}
\fD^s\check{\cY}(\mathbf{x},\hbar,q)\big|_{\alpha=0}=
\mathbf{x}^s\fD^s\check{F}(\mathbf{x}/\hbar,q),\qquad\hbox{where}\\
\fD^0\check{F}(w,q)=\frac{\check{F}(w,q)}{\check{I}_0(q)}\,,\quad
\fD^s\check{F}(w,q)=\frac{1}{\check{I}_s(q)}\left\{1+\frac{q}{w}\frac{d}{d q}\right\}
\fD^{s-1}\check{F}(w,q) \quad\forall\,s\!\in\!\ZZ^+,
\end{gather}
with $(\check{\cY},\check{F},\check{I})\!=\!(\cY, F,I), (\dot\cY,\dot F,\dot{I}),(\ddot\cY,\ddot F,\ddot{I})$.
\begin{theo}\cite[Theorem 4]{Zin2}\label{main1}If $\ell\!\in\!\ZZ^{\ge0}$, $n\!\in\!\ZZ^+$, and $\a\!\in\!(\ZZ^{>0})^{\ell}$ are such that
$|\a|=\sum_{i=1}^{\ell}\!=\!n$, then
\beq\label{main_1}
\dot{\cZ}(\hbar_1,\hbar_2,q) =\frac{1}{\hbar_1+\hbar_2}\sum_{s_1,s_2,r\ge0\atop s_1+s_2+r=n-1}(-1)^r\sigma_r\dot{\cZ}^{(s_1)}(\mathbf{x},\hbar_1,q)\ddot{\cZ}^{(s_2)}(\mathbf{x},\hbar_2,q),
\eeq
where $\sigma_r\in\QQ_{\alpha}$ is the $r$-th elementary symmetric polynomial in $\alpha_1,\cdots,\alpha_n$. There exists $\widetilde{\cC}_{s_1,s_2}^{(r)}\in\QQ[\alpha][[q]]$ such that
$$
\check{\cZ}^{(s)}(\mathbf{x},\hb,q)=\check{\cY}^{(s)}(\mathbf{x},\hb,q):=\sum_{r=0}^s\sum_{s'=0}^{s-r}
\widetilde{\cC}_{s-\ell^*(\a),s'-\ell^*(\a)}^{(r)}(q)\,\hb^{s-r-s'}\fD^{s'}\check{\cY}(\mathbf{x},\hb,q),
$$
where $(\check{\cZ},\ell^*)=(\dot\cZ,0),(\ddot\cZ,\ell)$, $\check{\cY}=\dot\cZ,\ddot\cZ$.
\end{theo}

\begin{lemm}\label{good}

\begin{eqnarray}
&&\bigg(\sum_{i=1}^n(\phi_{i}|_{p_{i}})^{-1}\bigg(q\frac{d}{dq}\cF_{n}^{(0,0)}(\alpha_i,q)+\alpha_i\bigg)\nonumber\\
&&\fR_{\hb_1=0}\left\{\fR_{\hb_2=0}
\left\{\mbox{e}^{-\cF_{n}^{(0,0)}(\alpha_i,q)(\frac{1}{\hb_1}+\frac{1}{\hb_2})}\widetilde{\cZ}_{ii}^*(\hb_1,\hb_2,q)
 \right\}\right\}\bigg)\bigg|_{\alpha=0}\nonumber\\
 &=&\bigg(\sum_{i=1}^n(\phi_{i}|_{p_{i}})^{-1}\bigg(q\frac{d}{dq}\cF_{n}^{(0,0)}(\alpha_i,q)+\alpha_i\bigg)\alpha_i^{n-1}
 \fR_{\hb_1=0}\bigg\{\fR_{\hb_2=0}\nonumber\\
&&\frac{1}{2(\hbar_1+\hbar_2)\hbar_1\hbar_2}
\bigg\{\mbox{e}^{-\cF_{n}^{(0,0)}(\alpha_i,q)(\frac{1}{\hb_1}+\frac{1}{\hb_2})}
\FF(\alpha_i/\hbar_1,\alpha_i/\hbar_2,q)\bigg\}\bigg\}\bigg)\bigg|_{\alpha=0},\nonumber
\end{eqnarray}

%\begin{eqnarray*}
%&&n\alpha_i^{n-1}+2\big(\hbar_1\!+\!\hbar_2\big)\hbar_1\hbar_2\widetilde{\cZ}_{ii}^*(\hbar_1,\hbar_2,q)\\
%&&-\alpha_i^{n-1}\FF(\alpha_i/\hbar_1,\alpha_i/\hbar_2,q)
%\in \mathcal{J}\tilde\QQ_i[\alpha]^{S_{n-1}}_{\hbar_1,\hbar_2}\big[\big[q\big]\big],
%\end{eqnarray*}
where
\begin{equation*}\begin{split}
\FF(w_1,w_2,q) &=\sum_{p=0}^{n-1-\ell}\frac{\bM^p\dot{F}(w_1,q)}{\dot I_p(q)}
\frac{\bM^{n-1-p}\ddot{F}(w_2,q)}{\ddot{I}_{n-1-p}(q)}\\
&\qquad +\sum_{p=1}^{\ell}\frac{\bM^{n-1+p}\ddot{F}(w_1,q)}{\ddot{I}_{n-1+p}(q)}
\frac{\bM^{n-p}\dot{F}(w_2,q)}{\dot{I}_{n-p}(q)}.
\end{split}\end{equation*}
\end{lemm}
\begin{proof}
By Theorem \ref{main1},
\begin{eqnarray*}
&&\sum_{s_1,s_2,r\ge0\atop s_1+s_2+r=n-1}(-1)^r\sigma_r\alpha^{s_1+s_2}_i+2(\hbar_1+\hbar_2)\hbar_1\hbar_2\widetilde{\cZ}^*_{ii} \\ &&=\sum_{s_1,s_2\ge0\atop s_1+s_2=n-1}\dot{\cY}^{(s_1)}(\alpha_i,\hbar_1,q)\ddot{\cY}^{(s_2)}(\alpha_i,\hbar_2,q)\\
&&+\sum_{s_1,s_2\ge0,r>0\atop s_1+s_2+r=n-1}(-1)^r\sigma_r\dot{\cY}^{(s_1)}(\alpha_i,\hbar_1,q)\ddot{\cY}^{(s_2)}(\alpha_i,\hbar_2,q).
\end{eqnarray*}
By the definition of $\dot{\cY}^{(s_1)}(\mathbf{x},\hbar_1,q)$ and $\ddot{\cY}^{(s_2)}(\mathbf{x},\hbar_1,q)$, and (\ref{xi-for}), the $q$ coefficient of
\begin{eqnarray*}
\fR_{\hb_1=0}\bigg\{\fR_{\hb_2=0}\frac{1}{2(\hbar_1+\hbar_2)\hbar_1\hbar_2}
\bigg\{\mbox{e}^{-\cF_{n}^{(0,0)}(\alpha_i,q)(\frac{1}{\hb_1}+\frac{1}{\hb_2})}&&\\
\sigma_r\dot{\cY}^{(s_1)}(\alpha_i,\hbar_1,q)\ddot{\cY}^{(s_2)}(\alpha_i,\hbar_2,q)\bigg\}\bigg\}&&
\end{eqnarray*}
satisfies the condition in Lemma \ref{good0}, for $r>0$.
Thus

\begin{eqnarray}
&&\bigg(\sum_{i=1}^n(\phi_{i}|_{p_{i}})^{-1}\bigg(q\frac{d}{dq}\cF_{n}^{(0,0)}(\alpha_i,q)+\alpha_i\bigg)\nonumber\\
&&\fR_{\hb_1=0}\left\{\fR_{\hb_2=0}
\left\{\mbox{e}^{-\cF_{n}^{(0,0)}(\alpha_i,q)(\frac{1}{\hb_1}+\frac{1}{\hb_2})}\widetilde{\cZ}_{ii}^*(\hb_1,\hb_2,q)
 \right\}\right\}\bigg)\bigg|_{\alpha=0}\nonumber\\
 &=&\bigg(\sum_{i=1}^n(\phi_{i}|_{p_{i}})^{-1}\bigg(q\frac{d}{dq}\cF_{n}^{(0,0)}(\alpha_i,q)+\alpha_i\bigg)\nonumber\\
 &&\fR_{\hb_1=0}\bigg\{\fR_{\hb_2=0}\frac{1}{2(\hbar_1+\hbar_2)\hbar_1\hbar_2}
\bigg\{\mbox{e}^{-\cF_{n}^{(0,0)}(\alpha_i,q)(\frac{1}{\hb_1}+\frac{1}{\hb_2})}\nonumber\\
&&
\sum_{s_1,s_2\ge0\atop s_1+s_2=n-1}\dot{\cY}^{(s_1)}(\alpha_i,\hbar_1,q)\ddot{\cY}^{(s_2)}(\alpha_i,\hbar_2,q)\bigg\}\bigg\}\bigg)\bigg|_{\alpha=0}.\nonumber
\end{eqnarray}
By (\ref{good4}) and (\ref{xi-for}), the $q$ coefficient of
\begin{eqnarray*}
&&\fR_{\hb_1=0}\bigg\{\fR_{\hb_2=0}\frac{1}{2(\hbar_1+\hbar_2)\hbar_1\hbar_2}
\bigg\{\mbox{e}^{-\cF_{n}^{(0,0)}(\alpha_i,q)(\frac{1}{\hb_1}+\frac{1}{\hb_2})}\nonumber\\
&&
\bigg(\sum_{s_1,s_2\ge0\atop s_1+s_2=n-1}\dot{\cY}^{(s_1)}(\alpha_i,\hbar_1,q)\ddot{\cY}^{(s_2)}(\alpha_i,\hbar_2,q)-
\alpha_i^{n-1}\FF(\alpha_i/\hbar_1,\alpha_i/\hbar_2,q)\bigg)\bigg\}\bigg\}
\end{eqnarray*}
%$$
%\fD^s\cY_{n;\a}(\mathbf{x},\hb,q)\big|_{\alpha=0}=
%\mathbf{x}^s\fD^sF_{n;\a}(\mathbf{x}/\hb,q)
%$$
%\beq
%\hbar^s\sum_{s'=0}^{\i}\sum_{r=0}^{s'}\cC_{s,s'}^{(r)}(q)\mathbf{x}^{s'-r}\hbar^{-s'}
%=\fD^s\cY_{n;\a}(\mathbf{x},\hbar,q)\,.\eeq
%$$
%\fD^{s_1}\check{\cY}(\mathbf{x},\hb,q)\fD^{s_2}\check{\cY}(\mathbf{x},\hb,q)=\mathbf{X}^{s_1+s_2}
%$$
%using
%$$
%\sum_{k\neq i}\frac{\alpha_i^m}{\alpha_i-\alpha_k}=0, \mbox{when} m\neq n-1
%$$
is of the form as in Lemma \ref{good0} and Lemma \ref{good'0}. Thus
\begin{eqnarray}
&&\bigg(\sum_{i=1}^n(\phi_{i}|_{p_{i}})^{-1}\bigg(q\frac{d}{dq}\cF_{n}^{(0,0)}(\alpha_i,q)+\alpha_i\bigg)\nonumber\\
&&\fR_{\hb_1=0}\left\{\fR_{\hb_2=0}
\left\{\mbox{e}^{-\cF_{n}^{(0,0)}(\alpha_i,q)(\frac{1}{\hb_1}+\frac{1}{\hb_2})}\widetilde{\cZ}_{ii}^*(\hb_1,\hb_2,q)
 \right\}\right\}\bigg)\bigg|_{\alpha=0}\nonumber\\
 &=&\bigg(\sum_{i=1}^n(\phi_{i}|_{p_{i}})^{-1}\bigg(q\frac{d}{dq}\cF_{n}^{(0,0)}(\alpha_i,q)+\alpha_i\bigg)\alpha_i^{n-1}
 \fR_{\hb_1=0}\bigg\{\fR_{\hb_2=0}\nonumber\\
&&
\frac{1}{2(\hbar_1+\hbar_2)\hbar_1\hbar_2}\bigg\{\mbox{e}^{-\cF_{n}^{(0,0)}(\alpha_i,q)(\frac{1}{\hb_1}+\frac{1}{\hb_2})}
\FF(\alpha_i/\hbar_1,\alpha_i/\hbar_2,q)\bigg\}\bigg\}\bigg)\bigg|_{\alpha=0}.\nonumber
\end{eqnarray}

\end{proof}
\begin{lemm}\cite[Lemma 5.4]{Popa}\label{good2}
\begin{eqnarray*}
\fR_{h_1=0}\fR_{h_2=0}\left\{
\frac{\e^{-\mu(q)\alpha_i(\hbar_1^{-1}+\hbar_2^{-1})}}{\hbar_1\hbar_2(\hbar_1\!+\!\hbar_2)}
\FF(\alpha_i/\hbar_1,\alpha_i/\hbar_2,q)\right\}&= \alpha_i^{-1}L(q)^{-1}q\frac{d}{dq}\AA(q),
\end{eqnarray*}
\end{lemm}
where
\begin{eqnarray*}
\AA(q)=\frac{n}{24}(n-1-2\sum_{r=1}^{\ell}\frac{1}{a_r})\mu(q)-\frac{3(n-1-\ell)^2+(n-2)}{24}\log(1-\a^{\a}q)\\
-\sum_{p=0}^{n-2-\ell}\binom{n-p-\ell}{2}\log \dot I_p(q).
\end{eqnarray*}
By Lemma \ref{good}, Lemma \ref{good2} and (\ref{good0})
\begin{eqnarray}\label{A-va}
&&\bigg(\sum_{i=1}^n\cA_i\bigg)\bigg|_{\alpha=0}\\
&=&\bigg(\sum_{i=1}^n(\phi_{i}|_{p_{i}})^{-1}\bigg(q\frac{d}{dq}\cF_{n}^{(0,0)}(\alpha_i,q)+\alpha_i\bigg)\nonumber\\
&&\fR_{\hb_1=0}\left\{\fR_{\hb_2=0}
\left\{\mbox{e}^{-\cF_{n}^{(0,0)}(\alpha_i,q)(\frac{1}{\hb_1}+\frac{1}{\hb_2})}\widetilde{\cZ}_{ii}^*(\hb_1,\hb_2,q)
 \right\}\right\}\bigg)\bigg|_{\alpha=0}\nonumber\\
 &=&\frac{1}{2}q\frac{d}{dq}\AA(q).\nonumber
\end{eqnarray}

\begin{theo}\label{main_3}
For projective complete intersection Calabi-Yau $X$,

$$
G_{1,0}=\frac{1}{2}\AA(q)+\frac{1}{24}\bigg(\sum_{i=1}^{\ell}(\frac{n}{a_i}-\binom{n}{2})\mu(q)-\frac{n(\ell+1)}{2}\log L(q)\bigg),
$$
where
\begin{eqnarray*}
\AA(q)=\frac{n}{24}(n-1-2\sum_{r=1}^{\ell}\frac{1}{a_r})\mu(q)-\frac{3(n-1-\ell)^2+(n-2)}{24}\log(1-\a^{\a}q)\\
-\sum_{p=0}^{n-2-\ell}\binom{n-p-\ell}{2}\log \dot I_p(q).
\end{eqnarray*}
\end{theo}
\begin{proof}
The proof is just combining Theorem \ref{The}, (\ref{A-va}) and (\ref{B-va}).
\end{proof}

\bibliographystyle{amsplain}

\end{document}